\documentclass[reqno,a4paper]{amsart}

\usepackage{amssymb}
\usepackage{amsmath}
\usepackage{amsthm}
\usepackage{latexsym}
\usepackage{mathrsfs}
\usepackage{eucal}
\usepackage{tipa}
\usepackage{graphicx}
\usepackage{hyperref}

\usepackage{color}

\allowdisplaybreaks[2]

\theoremstyle{plain}
\newtheorem{theorem}{Theorem}

\newtheorem{lemma}[theorem]{Lemma}
\newtheorem{corollary}[theorem]{Corollary}

\newtheorem{definition}[theorem]{Definition}
\newtheorem{remark}[theorem]{Remark}

\def\sDiv{\mathscr{D}}
\def\sCurl{\mathscr{C}}
\def\sCurlDagger{\mathscr{C}^\dagger}
\def\sTwist{\mathscr{T}}

\def\EMTensorH{\mathbf{G}}
\def\EMTensorH{\mathbf{H}}
\def\EMTensorV{\mathbf{V}}

\def\MorawetzCurrJ{\mathbf{P}}
\DeclareMathOperator{\tho}{\text{\rm \textthorn}}
\DeclareMathOperator{\edt}{\eth}

\def\SuperEnOne{\mathbf{SE}_1}

\newcommand{\vecMprimary}{A}

\newcommand{\scalMprimary}{q}

\newcommand{\OrthFrameT}{\widehat T}
\newcommand{\OrthFrameX}{\widehat X}
\newcommand{\OrthFrameY}{\widehat Y}
\newcommand{\OrthFrameZ}{\widehat Z}

\newcommand{\Reals}{\mathbb{R}}
\newcommand{\di}{d}

\newcommand{\EMTensorT}{T}

\newcommand{\GenVec}{\nu}

\newcommand{\SL}{\mathrm{SL}}

\newcommand{\Co}{{\mathbb C}}	      %the complex numbers
\newcommand{\eps}{\epsilon}

\numberwithin{equation}{section}

\begin{document}

\title{Decay of solutions to the Maxwell equation on the Schwarzschild background}

\author[L. Andersson]{Lars Andersson} \email{laan@aei.mpg.de}
\address{Albert Einstein Institute, Am M\"uhlenberg 1, D-14476 Potsdam,
  Germany
\and
Department of Mathematics, Royal Institute of Technology, SE-100 44 Stockholm, Sweden
}

\author[T. B\"ackdahl]{Thomas B\"ackdahl} \email{thobac@chalmers.se}
\address{The School of Mathematics  and the Maxwell Institute, University of Edinburgh, James Clerk Maxwell Building,
Peter Guthrie Tait Road, Edinburgh
EH9 3FD, UK
\and 
Mathematical Sciences, Chalmers University of Technology and University of Gothenburg, SE-412 96 Gothenburg, Sweden
}

\author[P. Blue]{Pieter Blue} \email{P.Blue@ed.ac.uk}
\address{The School of Mathematics and the Maxwell Institute, University of Edinburgh, James Clerk Maxwell Building,
Peter Guthrie Tait Road, Edinburgh
EH9 3FD,UK}

\begin{abstract}
A new Morawetz or integrated local energy decay estimate for Maxwell test fields on the exterior of a Schwarzschild black hole spacetime is proved. The proof makes use of a new superenergy tensor $\EMTensorH_{ab}$ defined in terms of the Maxwell field and its first derivatives. The superenergy tensor, although not conserved, yields a conserved higher order energy current $\EMTensorH_{ab} (\partial_t)^b$. The tensor $\EMTensorH_{ab}$ vanishes for the static Coulomb field, and the Morawetz estimate proved here therefore yields integrated decay for the Maxwell field to the Coulomb solution on the Schwarzschild exterior.
\end{abstract}

\maketitle

\section{Introduction}
The exterior region of the Schwarzschild spacetime is given in
Schwarzschild coordinates by
$(x^a) = (t,r,\theta,\phi)\in
\Reals \times (2M,\infty) \times [0,\pi] \times [0,2\pi)$
with metric
\begin{align*}
g_{ab}\di x^a \di x^b ={}&(1-2M/r)\di t^2 -(1-2M/r)^{-1}\di r^2
-r^2(\di\theta^2+\sin^2\theta\di\phi^2) .
\end{align*}
The exterior region can be extended, but for simplicity we will not
treat the extension here.  The Schwarzschild metric $g_{ab}$ is static
and spherically symmetric. In particular, the vector field $\xi^a =
(\partial_t)^a$ is Killing, timelike in the exterior, and orthogonal
to the level sets of $t$. With our choice of signature,\footnote{In
  this paper, we follow the conventions of \cite{Penrose:1986fk} and
  assume all fields are smooth unless otherwise stated.} the timelike
condition is that $\xi^a \xi_a > 0$. In fact, as $r\to\infty$, we find
$\xi^a\xi_a\to1$. Recall the condition for $\xi^a$ to be a Killing
vector is $\nabla_{(a} \xi_{b)} = 0$. A related condition for a
$2$-form $Y_{ab}$ to be a Killing-Yano tensor is $\nabla_{(a} Y_{b)c}
= 0$.  The Schwarzschild exterior admits a Killing-Yano tensor given
by
\begin{align}
Y_{ab}={} - 2 r^3 \sin\theta d\theta_{[a}d\phi_{b]}.  \label{eq:Ydef}
\end{align}

A Maxwell test field is a real 2-form $F_{ab}$ satisfying the equations
$\nabla^a F_{ab} = 0, \nabla_{[a} F_{bc]} = 0$.
From now on, unless otherwise stated we shall assume that $F_{ab}$ is a Maxwell field.

To state our main result, we introduce the real $1$-form
$$
U_a = -r^{-1}\nabla_a r ,
$$
and define in terms of $U_a$, $F_{ab}$ and $Y_{ab}$, a real $2$-form
$Z_{ab}$, and a complex $1$-form $\beta_a$ by
\begin{align}
Z_{ab}={}&- \tfrac{4}{3} {*}F_{[a}{}^{c}Y_{b]c},
\label{eq:DefZTensorVersion}\\
\beta_a={}&- \tfrac{1}{2} U^{b} Z_{ab}
 -  \tfrac{1}{2}i U^{b} {*}Z_{ab}
 -  \tfrac{1}{2} \nabla_{b}Z_{a}{}^{b}
 -  \tfrac{1}{2}i \nabla_{b}{*}Z_{a}{}^{b} .
\label{eq:DefbetaTensorVersion}
\end{align}
The superenergy tensors (in the terminology of \cite{senovilla:2000CQGra..17.2799S})  for $\beta_a$ and $Z_{ab}$ are
\begin{align}
\EMTensorH_{ab}={}&\beta_{(a}\bar{\beta}_{b)}
 - \tfrac{1}{2} g_{ab} \beta^{c} \bar{\beta}_{c} ,\\
\mathbf{W}_{ab}={}&- \tfrac{1}{2} Z_{a}{}^{c} Z_{bc} + \tfrac{1}{8} g_{ab} Z_{cd} Z^{cd}.
\end{align}
\begin{remark} The tensors $\EMTensorH_{ab}$, $\mathbf{W}_{ab}$ are quadratic in $F_{ab}$ and its derivatives up to first order (zeroth order for $\mathbf{W}_{ab}$). Recall that a tensor $S_{ab}$ satisfies the dominant energy condition if $S_{ab} \nu^a \zeta^b \geq 0$ for future-directed timelike vectors $\nu^a$, $\zeta^a$. It follows from their definition that both $\EMTensorH_{ab}$ and $\mathbf{W}_{ab}$ satisfy the dominant energy condition, cf. \cite{bergqvist:1999CMaPh.207..467B,senovilla:2000CQGra..17.2799S}. The notion of superenergy tensor does not necessarily imply any conservation property.
\end{remark}
Given a vector field $\GenVec^a$, the superenergy current defined by $\EMTensorH_{ab}$ with respect to $\GenVec^a$  is $\EMTensorH_{ab} \GenVec^b$.
For the static Killing field $\xi^a$ we have that the current $\EMTensorH_{ab} \xi^b$ is conserved,
\begin{equation}\label{eq:Jcons}
\nabla^a \left (\EMTensorH_{ab} \xi^b\right ) = 0,
\end{equation}
see Lemma~\ref{lem:divH}.  In contrast to the standard symmetric
energy-momentum tensor for the Maxwell field, the tensors
$\EMTensorH_{ab}$ and $\mathbf{W}_{ab}$ are \emph{not} conserved in
general, $\nabla^a \EMTensorH_{ab} \ne 0$, $\nabla^a \mathbf{W}_{ab}
\ne 0$. On spacetimes admitting a Killing spinor satisfying an aligned matter
condition, we have introduced another conserved tensor
\cite{AnderssonBackdahlBlue:NewConservationLaw:2014arXiv1412.2960A}.

For a vector field $\nu^a$, a spacelike hypersurface $\Sigma$ with
future-directed normal $N^a$, and induced volume element $\mu_\Sigma$, we define
\begin{align}
E_\GenVec(\Sigma)={}& \int_\Sigma
\EMTensorH_{ab} \GenVec^b
N^a \di\mu_{\Sigma} . \label{eq:Eflux}
\end{align}
It follows from the dominant energy condition that if
$\GenVec^a$ is timelike and future-directed, then $E_\GenVec(\Sigma)\geq 0$. By the conservation identity \eqref{eq:Jcons}, if $\Sigma_1, \Sigma_2$ are hypersurfaces bounding a spacetime region $\Omega$, then the superenergy with respect to $\xi^a$ on $\Sigma_1$ and $\Sigma_2$ are equal.

Let $\OrthFrameT^a, \OrthFrameX^a, \OrthFrameY^a, \OrthFrameZ^a$ be the orthonormal frame
\begin{align*}
\OrthFrameT^a={}&(1-2M/r)^{-1/2}(\partial_t)^a,&
\OrthFrameX^a={}&r^{-1}(\partial_\theta)^a,
\\
\OrthFrameY^a={}&r^{-1}\csc\theta(\partial_\phi)^a,&
\OrthFrameZ^a={}&(1-2M/r)^{1/2}(\partial_r)^a,
\end{align*}
adapted to the foliation by level sets of $t$. Then if $\Sigma_t$ is a level set of $t$, the  superenergy
with respect to $\xi^a$ on $\Sigma_t$ is given by
\begin{align}
E_{\xi}(\Sigma_t) ={}& \frac{1}{2} \int_{\Sigma_t}
\left (|\beta_{\OrthFrameT}|^2+|\beta_{\OrthFrameX}|^2+|\beta_{\OrthFrameY}|^2+|\beta_{\OrthFrameZ}|^2\right ) r^2\sin\theta\di r\di\theta\di\phi , \label{eq:Exi}
\end{align}
where for a vector $\GenVec^a$, $\beta_\GenVec = \beta_a \GenVec^a$. Similarly,
$$
\mathbf{W}_{\OrthFrameT\OrthFrameT} = \tfrac{1}{4} \bigl(|Z_{\OrthFrameT\OrthFrameX}|^2 + | Z_{\OrthFrameT\OrthFrameY}|^2 + |Z_{\OrthFrameX\OrthFrameZ}|^2 + |Z_{\OrthFrameY\OrthFrameZ}|^2\bigr).
$$
Hence, it controls all components of $Z_{ab}$, because the structure of $Y_{ab}$ gives $Z_{\OrthFrameT\OrthFrameZ}=0$, $Z_{\OrthFrameX\OrthFrameY}=0$, see discussion below.

We are now ready to state our main result.
\begin{theorem}[Energy bound and Morawetz estimate]
\label{thm:MainResult}
Let $\Sigma_1$ and $\Sigma_2$ be spherically symmetric spacelike
hypersurfaces in the exterior region of the Schwarzschild spacetime
such that $\Sigma_2$ lies in the future of $\Sigma_1$ and
$\Sigma_2\cup -\Sigma_1$ is the oriented boundary of a spacetime
region $\Omega$.

If the real $2$-form $F_{ab}=F_{[ab]}$
is a solution of the Maxwell equations
\begin{align}
\nabla^a F_{ab}&=0,\quad
\nabla_{[c}F_{ab]}=0,
\label{eq:MaxwellTensorVersion}
\end{align}
on the Schwarzschild exterior,
and $Z_{ab}$ and $\beta_a$
are defined by equations
\eqref{eq:DefZTensorVersion}-\eqref{eq:DefbetaTensorVersion}, then
\begin{align}
E_{\xi}(\Sigma_2)={}&E_{\xi}(\Sigma_1),\\
\int_{\Omega} \vert \beta_{a}\vert^2_{1,\text{deg}}
+\frac{2 M}{25 r^4}\vert Z_{ab}\vert^2_{2} d\mu_{\Omega}
\leq{}& \frac{72}{5}E_{\xi}(\Sigma_1), \label{eq:MorIneq}
\end{align}
where $E_\xi(\Sigma_i)$ is the flux \eqref{eq:Eflux} associated with $\xi^a$, and
$|\beta_a|_{1,deg}$ and $|Z_{ab}|_{2}$ are, respectively, the degenerate $1$-form norm of $\beta_a$ and the $2$-form norm of $Z_{ab}$ defined by

\begin{subequations}
\begin{align*}
\vert \beta_{a}\vert^2_{1,\text{deg}}={}&
\frac{(r - 3 M)^2}{r^3}\bigl(|\beta_{\OrthFrameX}|^2 + |\beta_{\OrthFrameY}|^2\bigr)  + \frac{M	 (r - 2 M)}{r^3}|\beta_{\OrthFrameZ}|^2
 + \frac{M  (r - 3 M)^2 (r - 2 M)}{r^5}|\beta_{\OrthFrameT}|^2,\\
\vert Z_{ab}\vert^2_{2}={}& \frac{ (r - 2 M)}{r} \mathbf{W}_{\OrthFrameT\OrthFrameT}.
\end{align*}
\end{subequations}
\end{theorem}
\begin{remark}
The approach presented here can be extended to the inhomogenous Maxwell equations.
%The extra terms that appear can be estimated with Cauchy-Schwarz inequalities based on the norms $|\cdot|_{1,\text{deg}}$ and $|\cdot|_{2}$.
\end{remark}
 Estimates of the form \eqref{eq:MorIneq} are called Morawetz or integrated local energy decay estimates, since they show that local $L^2$ norms in space are integrable in time. Since the right-hand side of inequality \eqref{eq:MorIneq} depends only on initial data, it follows from the Morawetz estimate that the energy in stationary regions must tend sequentially to zero for large times, which demonstrates that the field disperses. In the present discussion the energy refers to the superenergies defined in terms of $\beta_a$ and $Z_{ab}$, and dispersion refers to the fact that these fields tend to
zero asymptotically in any stationary region. Although the parts of the Maxwell field controlled by $Z_{ab}$ disperse, the non-radiating Coulomb solution, which is a bound state, remains. Below, we discuss that the fields $\beta_a$ and $Z_{ab}$ vanish when evaluated on the Coulomb solution.

Recall that the standard Maxwell energy-momentum tensor, given up to a constant factor by
\begin{align*}
T_{ab} = - \tfrac{1}{2} F_{a}{}^{c} F_{bc} + \tfrac{1}{8} g_{ab} F_{cd} F^{cd},
\end{align*}
is traceless and conserved if $F_{ab}$ is a Maxwell field, see \cite[Chapter 5]{Penrose:1986fk}.
We remark that the analogue of Theorem~\ref{thm:MainResult} cannot hold if we replace $\EMTensorH_{ab}$ by $T_{ab}$, due to the fact that the energy density defined in terms of $T_{ab}$ is non-zero when evaluated on the Coulomb solution. In fact, for a radial Morawetz vector field, it is not difficult to see that the bulk term $\nabla^a (T_{ab} A^b) = T_{ab}\nabla^a A^b$ must change sign at $r=3M$, see \cite[section 7]{ABB:jubilee}.

The main motivation for this work is to develop techniques that might
be useful in approaching the Kerr stability conjecture. The Kerr
family of spacetimes is a two-parameter family of solutions of the
Einstein equation describing rotating black holes. The
Schwarzschild spacetimes make up a one-parameter subfamily corresponding
to zero rotation speed. The Kerr stability conjecture is that the Kerr
family of solutions is asymptotically stable under the Einstein
equation, although the Schwarzschild subfamily in isolation is not
expected to be stable, except when restricting to the axially symmetric case with zero angular momentum. A natural strategy is to understand solutions
of first the null geodesic equation, the wave equation, Maxwell equation,
and then some reasonable linearisation of the Einstein equation,
before approaching the full nonlinear Einstein equation.

To compare Theorem~\ref{thm:MainResult} and its proof with earlier
work in the literature, one must consider the null or spinor
decomposition of a $2$-form. To do so, one first introduces a complex
null tetrad
$(l^a, n^a, m^a, \bar{m}^a)$
normalized so that $g_{ab}l^a n^b=-g_{ab}m^a\bar{m}^b=1$ and all other inner
products are zero. For example, in the exterior of the Schwarzschild
spacetime, one can use
\begin{align}
l^a={}&\frac{1}{\sqrt{2}}(\OrthFrameT^a+\OrthFrameZ^a),&
n^a={}&\frac{1}{\sqrt{2}}(\OrthFrameT^a-\OrthFrameZ^a),&
m^a={}&\frac{1}{\sqrt{2}}(\OrthFrameX^a+i\OrthFrameY^a).
\label{eq:SchwarzschildTetrad}
\end{align}
This is a principal null tetrad.
For a $2$-form $F_{ab}$, the Newman-Penrose scalars
with respect to a null tetrad $l^a, n^a, m^a, \bar{m}^a$ (not necessarily principal) are
$$
\phi_{0} ={} F_{ab}l^am^b, \quad \phi_{1}={}
\frac{1}{2} ( F_{ab}l^an^b+F_{ab}\bar{m}^am^b) , \quad
\phi_{2}={} F_{ab}\bar{m}^a n^b .
$$
The components $\phi_{0}, \phi_{2}$ are called the extreme components and $\phi_{1}$ is the middle component. The extreme components depend on the scaling of the tetrad, while the middle component is a true scalar.
Since the vector fields
$\partial_\theta$ and $\csc\theta\partial_\phi$ are not
continuous at $\theta\in\{0,\pi\}$, the extreme components are not
smooth as functions  even if $F_{ab}$ is smooth, but are smooth when viewed as sections of an appropriate complex line bundle, cf. \cite{Penrose:1986fk}, see also \cite{Blue:Maxwell}.

A key tool in most previous work on this topic has been the existence of
second-order, decoupled equations for individual null components. In
the exterior of the Schwarzschild spacetime, the middle component $\phi_{1}$ of the Maxwell field
satisfies (after rescaling by a power of $r$) a wave equation with real potential, which is the Schwarzschild case of the Fackerell-Ipser equation \cite{FackerellIpser}.
In the Kerr spacetime, the potential in the Fackerell-Ipser equation is complex, and the extreme
components satisfy a more complicated, complex second-order hyperbolic
PDE known as the Teukolsky equation \cite{Teukolsky}. Both the
Fackerell-Ipser and Teukolsky equations have analogues arising in
certain linearisations of the Einstein equation, making them
particularly interesting as part of the strategy for approaching the
Kerr stability conjecture. In fact, the potential in the Fackerell-Ipser wave equation on the exterior of the Schwarzschild spacetime is a multiple of the potential in the Regge-Wheeler equation governing odd parity gravitational perturbations of the Schwarzschild metric, cf.  \cite{Price, donninger:schlag:soffer:2009arXiv0911.3179D}.

The Maxwell equation in the exterior of the Schwarzschild spacetime
has stationary solutions given by
\begin{align*}
F_{ab}&= 4\bigl (q_E \OrthFrameT_{[a}\OrthFrameZ_{b]} +q_B\OrthFrameX_{[a}\OrthFrameY_{b]}\bigr) r^{-2}.
\end{align*}
For these, the Maxwell scalars are respectively $\phi_{0} = 0$, $\phi_{1} = (q_E+iq_B)/r^2$,
and $\phi_{2} = 0$.
These are called the Coulomb solutions and are the only
solutions that do not decay to zero. Thus, one commonly says the
Coulomb solutions are supported in the middle component. Clearly,
these must be excluded for a decay estimate to hold.

At each point, the middle component of the $2$-form $Z_{ab}$ is zero
and its extreme components (modulo a sign) are $r$ times the
corresponding extreme component of the complex self-dual $2$-form
$F_{ab}+i(*F)_{ab}$. Thus, by working with $Z_{ab}$, we have
introduced a variable that geometrically excludes the Coulomb
solutions. Since at any point the components of $F_{ab}$ can
be freely specified while constraining $F_{ab}$ to be smooth and compactly
supported on any spacelike hypersurface $\Sigma_1$, the estimate in
Theorem~\ref{thm:MainResult} is non-vacuous.

Decay for the Maxwell equation in the exterior region of a
Schwarzschild spacetime or its generalisations has been studied in
\cite{finster:smoller:MR2471853,Blue:Maxwell,donninger:schlag:soffer:2009arXiv0911.3179D,sterbenz:tataru:2013arXiv1305.5261S,andersson:blue:maxwell:2013arXiv1310.2664A,MTT:Maxwell:2014arXiv1411.3693M}.
It has long been clear that, since $\xi^a$ is a time-like Killing
vector, there is a conserved positive energy for $F_{ab}$, and that
this can be used as a foundation for subsequent decay estimates. All
works have used both a second-order PDE for certain components and the
fact that such solutions arise from the original first-order Maxwell
equation. The paper \cite{finster:smoller:MR2471853} is the only
previous work we know of to prove decay using the Teukolsky equation;
the conclusion of that work is that at fixed $r$ the extreme components go to
zero, and the method relies on an integral representation\footnote{We have been unable to verify the claims in \cite{finster:smoller:MR2471853} concerning the spin-2 case. In particular we have been unable to verify that the energy presented in \cite[Eq. (1.4)]{finster:smoller:MR2471853}, also discussed in appendix A of that paper, is conserved. For the Maxwell case, the analogous claim follows from the conserved positive energy for $F_{ab}$.}.
In contrast, the remaining work has focused on the middle
component, which satisfies the Fackerell-Ipser equation and which
requires the contribution from the Coulomb solution to be
explicitly subtracted off.

Since the Fackerell-Ipser equation is a wave-like equation, we recall
that the last 15 years have seen the proof of Morawetz and related
decay estimates for the wave equation in the exterior of a
Schwarzschild black hole
\cite{BlueSoffer:LongPaper,BlueSoffer:Errata,BlueSterbenz,DafermosRodnianski:RedShift}
essentially using vector-field methods and following ideas introduced
for a model problem in \cite{LabaSoffer}. In the Kerr case, because of
the ergoregion generated by rotation, energy and Morawetz estimates
had to be proved simultaneously instead of sequentially. This was
first done in the very slowly rotating case
\cite{DafermosRodnianski:KerrEnergy,TataruTohaneanu,AnderssonBlue:KerrWave}
and more recently the full range of subextremal Kerr black holes
\cite{DafermosRodnianskiShlopentokhRothman}. Because of the
complicated nature of the orbiting null geodesics around Kerr black
holes, it was necessary to extend the vector-field method. In
\cite{DafermosRodnianski:KerrEnergy,DafermosRodnianskiShlopentokhRothman}
and \cite{TataruTohaneanu}, this was done by blending it with
separation-of-variable and operator-theory techniques respectively;
whereas in \cite{AnderssonBlue:KerrWave}, the first and third authors
were able to work with purely differential operators. This involved
using hidden symmetries, second-order differential operators that take
solutions to solutions but which are not decomposable into first-order
symmetries. In \cite{AnderssonBackdahlBlue}, we have explored the
analogue for the Maxwell equation.

Returning to the Maxwell problem, \cite{Blue:Maxwell} proves a
Morawetz estimate for solutions to the Fackerell-Ipser equation on the
exterior Schwarzschild spacetime and uses this to prove a Morawetz estimate and $t^{-1}$ decay
for the components of $F_{ab}$ at fixed $r$ and a hierarchy of decay
rates along null infinity (where $t\rightarrow\infty$ but $t-r$
remains bounded), and relies upon energies arising from the
vector-field method. The paper
\cite{donninger:schlag:soffer:2009arXiv0911.3179D} proves $t^{-4}$
decay at fixed $r$ for the middle component by summing Sigal-Soffer
propagation estimates for separated modes, and
\cite{sterbenz:tataru:2013arXiv1305.5261S} treats inhomogeneous
Maxwell equations on a very general class of stationary, spherically
symmetric black-hole spacetimes and the results extend beyond the
exterior region. The proof uses energies and vector-field methods
first to prove a Morawetz estimate for the middle component and then
to extend this to a Morawetz estimate for all components. The paper
\cite{MTT:Maxwell:2014arXiv1411.3693M} proves a ``black-box'' result,
in which a Morawetz estimate of the strong form proved in
\cite{sterbenz:tataru:2013arXiv1305.5261S} is assumed and shown to be
sufficient to imply a $t^{-3}$ decay rate at fixed $r$ without using
the details of the proof of the Morawetz estimate.
The paper \cite{ghanem:2014arXiv1409.8040G} similarly proves a decay result  based on the assumption of a Morawetz estimate.

In \cite{andersson:blue:maxwell:2013arXiv1310.2664A}, the first and third authors proved a
Morawetz estimate outside a very slowly rotating (and hence
non-spherically symmetric) Kerr black hole. So far, it has not been
possible to prove decay estimates for the Maxwell equation without
passing to a second-order equation. Energies for second-order equations
are naturally at the $H^1$ level rather than the $L^2$ level for
first-order equations. In the Schwarzschild case, the Fackerell-Ipser equation has a conserved, non-negative energy, so there was no
obstacle.
However, for the Fackerell-Ipser equation on the Kerr spacetime, the potential is complex, which destroys the conservation structure. This could be
compensated for, but only by introducing fractional derivative
operators (more precisely, fractional powers of the separation constants
after separation of variables). It seems reasonable to imagine that
the original $L^2$ type energy could equally well be combined with a
fractional derivative operator to prove a similar result.

In this paper, we develop a new method for treating the Maxwell equation which does not require introducing a second-order equation. This opens the possibility of applying the differential hidden symmetry operators from \cite{AnderssonBackdahlBlue}, which might allow for the proof of decay estimates in the Kerr case without using fractional derivative operators or fractional powers of separation constants.

The principal variables in this paper are $Z_{ab}$ and $\beta_a$, which are
constructed from the extreme components, which satisfy the Teukolsky
equations, rather than the middle component, which in the case we consider  satisfies the Schwarzschild case of the
Fackerell-Ipser wave equation.
We believe this has two major advantages.
First, it provides a local way to exclude the Coulomb solutions,
without integrating over spheres. Second, we have recently found a new
conservation law for the Maxwell equation
\cite{AnderssonBackdahlBlue:NewConservationLaw:2014arXiv1412.2960A} which generates $H^1$
level energies. Curiously, although $Z_{ab}$ in this paper and
in \cite{AnderssonBackdahlBlue:NewConservationLaw:2014arXiv1412.2960A} are the same, the
auxiliary $1$-forms $\beta_a$ and $\eta_a$ and the associated
symmetric $2$-tensors $\EMTensorH_{ab}$ and $\EMTensorV_{ab}$
differ. Nonetheless, in the sense given in
\cite{AnderssonBackdahlBlue:NewConservationLaw:2014arXiv1412.2960A}, the leading-order
parts of $\EMTensorH_{ab}$ and the conserved $\EMTensorV_{ab}$ are the
same. We are currently investigating how to replace $\EMTensorH_{ab}$
by $\EMTensorV_{ab}$ so that the argument in this paper can be extended to the rotating Kerr case.

By applying the Maxwell equations, the variable $\beta_a$ can be transformed from an expression in the derivatives of the extreme components to an expression purely in terms of the middle component. Recall that in GHP notation, the Maxwell field, $F_{ab}$, is decomposed into $\phi_0$, $\phi_1$, $\phi_2$, and we refer to $\phi_1$ as the middle component. In particular, $\beta_a=\nabla_a \Upsilon - U_a \Upsilon$, where $\Upsilon$ is $r$ times the middle component of $F_{ab}$. Similarly, one might view $\EMTensorH_{ab}$ as a type of energy-momentum tensor (albeit not conserved) for the Fackerell-Ipser equation. Following this, one might be tempted to conclude that Theorem~\ref{thm:MainResult} follows immediately from the Maxwell equation and the results about the middle component in \cite{Blue:Maxwell} or \cite{sterbenz:tataru:2013arXiv1305.5261S}. We emphasise that such a naive interpretation is false. First, Theorem~\ref{thm:MainResult} includes undifferentiated factors of $\phi_0$ and $\phi_2$, which arise from the components of $Z_{ab}$ and which cannot be recovered from the middle component using local operators.  Second, in \cite{Blue:Maxwell}, a Hardy estimate in the radial derivative was applied to the middle component; in Section~\ref{sec:Hardy}, a Hardy estimate in angular derivatives is applied to the extreme components. Morawetz estimates typically require a Hardy estimate to obtain a globally positive coefficient for the undifferentiated terms.

The argument in this paper very closely follows the argument used for
the wave equation. From a vector field $\vecMprimary^a$ and a scalar
field $\scalMprimary$, we define the a current $\MorawetzCurrJ_a$ in terms of $Z_{ab}$ and $\beta_a$ by\footnote{This can be compared with the corresponding expression for the wave equation,
\begin{align*}
\MorawetzCurrJ_a=\EMTensorT[u]_{ab}\vecMprimary^{b}+\scalMprimary
  u\nabla_a u -(\nabla_a\scalMprimary)u^2/2,
\end{align*} where
  $\EMTensorT[u]_{ab}$ is the energy-momentum tensor for the wave
equation. Thus, we use $\EMTensorH_{ab}$ in place of the wave
equation's energy-momentum tensor, and we use the terms involving $q$ in   \eqref{eq:MorawetzMomentum}
in place of $\scalMprimary u\nabla_a u
-(\nabla_a\scalMprimary)u^2/2$.}
\begin{align}
\MorawetzCurrJ_a={}&\EMTensorH_{ab} \vecMprimary^{b}
+ \tfrac{1}{4} \scalMprimary \beta^{b} Z_{ab}
 + \tfrac{1}{4} \scalMprimary \bar{\beta}^{b} Z_{ab}
 -  \tfrac{1}{4}i \scalMprimary \beta^{b} {*}Z_{ab}
 + \tfrac{1}{4}i \scalMprimary \bar{\beta}^{b} {*}Z_{ab}
\nonumber \\%*
&+ \tfrac{1}{16} Z_{bc} Z^{bc} \nabla_{a}\scalMprimary
  -  \tfrac{1}{4} Z_{a}{}^{c} Z_{bc} \nabla^{b}\scalMprimary .
\label{eq:MorawetzMomentum}
\end{align}

We now choose
\begin{subequations}
\begin{align}
\vecMprimary^a={}&\frac{(r - 3 M) (r - 2 M)}{2 r^2}(\partial_r)^a, \label{eq:Adef} \\
\scalMprimary ={}&\frac{9 M^2 (r - 2 M) (2 r - 3 M)}{4 r^5}. \label{eq:qdef}
\end{align}
\end{subequations}
The motivation for this choice is explained in the proof of
Theorem~\ref{thm:MorawetzEst}.
The vector field $\vecMprimary^a$ is the same as occurs in
the analysis of the wave equation.
%This is to be expected, since it is chosen to have the geometric property that it changes sign at $r=3M$,
%the unique radius at which there are orbiting null geodesics in the Schwarzschild spacetime.
In particular, it has the geometric property that it points away from the orbiting null geodesics, which in the Schwarzschild spacetime are located at $r = 3M$.

For any spacelike hypersurface $\Sigma$, we define
\begin{align}
E_\xi(\Sigma)={}&\int_{\Sigma}	\EMTensorH_{ab}\xi^{b}
N^{a} \mbox{d}\mu_{\Sigma}, \label{eq:Exidef}\\
E_{\xi+\vecMprimary,\scalMprimary}(\Sigma)
={}&\int_{\Sigma}  \left(\EMTensorH_{ab}\xi^b+\MorawetzCurrJ_a\right)
N^{a} \mbox{d}\mu_{\Sigma}. \label{eq:ExiAqdef}
\end{align}
As discussed above,
$E_{\xi}(\Sigma)$ is nonnegative and conserved.
In Section~\ref{sec:posenergy} we show that on spacelike, spherically symmetric hypersurfaces, the energies $E_\xi(\Sigma)$ and $ E_{\xi +\vecMprimary, \scalMprimary}(\Sigma) $ are uniformly equivalent.
In Section~\ref{sec:integrateddecay}, we show that the integral over any spherically
symmetric region $\Omega$ of $-\nabla_a\MorawetzCurrJ^a$ dominates the integral on the left hand side of \eqref{eq:MorIneq}. The proof of this fact relies on a spherical Hardy estimate, see  Lemma~\ref{lem:Hardy-application}.
Finally these facts are combined in Section~\ref{sec:mainthmproof} to yield a proof of Theorem \ref{thm:MainResult}.

\section{Preliminaries} \label{sec:prel}
For the remainder of this paper, we will make use of the 2-spinor
formalism, following the conventions of \cite{Penrose:1986fk}.
The exterior Schwarzschild spacetime is oriented and globally hyperbolic and hence also spin.

The spin group is $\SL(2,\Co)$ which has the inequivalent spinor representations $\Co^2$ and $\bar{\Co}^2$.  Unprimed upper case latin indices and their primed versions are used for sections of the corresponding spinor bundles, respectively. The correspondence between spinors and tensors makes it possible to translate all tensor expressions to spinor form. The action of $\SL(2,\Co)$ on $\Co^2$ leaves invariant the spin metric $\eps_{AB} = \eps_{[AB]}$, which is used to raise and lower indices on tensors.  The metric $g_{ab}$ is related to $\eps_{AB}$ by $g_{ab} = \eps_{AB} \bar\eps_{A'B'}$, and the scalar and Weyl curvatures are represented by the spinor fields $\Lambda, \Psi_{ABCD}$, respectively.
We denote  the
space of symmetric spinors with $k$ unprimed indices and $l$
primed indices by $\mathcal{S}_{k,l}$.

The principal null tetrad $l^a, n^a, m^a, \bar{m}^a$ corresponds to a principal spin dyad $o_A, \iota_A$ via $l_a = o_A \bar{o}_{A'}, n_a = \iota_A \bar{\iota}_A, m_a = o_A \bar{\iota}_{A'}, \bar{m}_a = \iota_A \bar{o}_{A'}$. Given a symmetric spinor field $\varphi_{A \cdots D A' \cdots D'}$,  we denote the dyad components by $\phi_{ii'}$ where $i,i'$ denote the number of contractions with $\iota^A, \bar{\iota}^{A'}$, respectively. For example, a 1-form $\varphi_{AA'}$, has components $\varphi_{00'} = \varphi_{AA'} o^A \bar{o}^{A'}, \cdots, \varphi_{11'} = \varphi_{AA'} \iota^{A}\bar{\iota}^{A'}$.

The 2-spinor formalism makes it possible to systematically exploit
decompositions in terms of irreducible representations of the spin
group $\SL(2,\Co)$. For the spinor calculations, we have made use of the \emph{SymManipulator} package \cite{Bae11a}, developed by one of the
authors (T.B.) for the Mathematica based symbolic differential
geometry suite \emph{xAct} \cite{xAct}.

\subsection{Fundamental operators}
\begin{definition}[\protect{\cite[Definition 13]{AnderssonBackdahlBlue}}]
For any $\varphi_{A_1\dots A_k}{}^{A_{1}'\dots A_{l}'}\in \mathcal{S}_{k,l}$, we define the operators
$\sDiv_{k,l}:\mathcal{S}_{k,l}\rightarrow \mathcal{S}_{k-1,l-1}$,
$\sCurl_{k,l}:\mathcal{S}_{k,l}\rightarrow \mathcal{S}_{k+1,l-1}$,
$\sCurlDagger_{k,l}:\mathcal{S}_{k,l}\rightarrow \mathcal{S}_{k-1,l+1}$ and
$\sTwist_{k,l}:\mathcal{S}_{k,l}\rightarrow \mathcal{S}_{k+1,l+1}$ as
\begin{align*}
(\sDiv_{k,l}\varphi)_{A_1\dots A_{k-1}}{}^{A_1'\dots A_{l-1}'}\equiv{}&
\nabla^{BB'}\varphi_{A_1\dots A_{k-1}B}{}^{A_1'\dots A_{l-1}'}{}_{B'},\\
(\sCurl_{k,l}\varphi)_{A_1\dots A_{k+1}}{}^{A_1'\dots A_{l-1}'}\equiv{}&
\nabla_{(A_1}{}^{B'}\varphi_{A_2\dots A_{k+1})}{}^{A_1'\dots A_{l-1}'}{}_{B'},\\
(\sCurlDagger_{k,l}\varphi)_{A_1\dots A_{k-1}}{}^{A_1'\dots A_{l+1}'}\equiv{}&
\nabla^{B(A_1'}\varphi_{A_1\dots A_{k-1}B}{}^{A_2'\dots A_{l+1}')},\\
(\sTwist_{k,l}\varphi)_{A_1\dots A_{k+1}}{}^{A_1'\dots A_{l+1}'}\equiv{}&
\nabla_{(A_1}{}^{(A_1'}\varphi_{A_2\dots A_{k+1})}{}^{A_2'\dots A_{l+1}')}.
\end{align*}
\end{definition}
The operator $\sDiv_{k,l}$ only makes sense when $k\geq 1$ and $l\geq 1$. Likewise $\sCurl_{k,l}$ is
defined only if $l \geq 1$ and	$\sCurlDagger_{k,l}$ only if $k\geq 1$.
From the definition it is clear that the complex conjugates of
$(\sDiv_{k,l}\varphi)$,
$(\sCurl_{k,l}\varphi)$,
$(\sCurlDagger_{k,l}\varphi)$ and
$(\sTwist_{k,l}\varphi)$ are
$(\sDiv_{l,k}\bar\varphi)$,
$(\sCurlDagger_{l,k}\bar\varphi)$,
$(\sCurl_{l,k}\bar\varphi)$ and
$(\sTwist_{l,k}\bar\varphi)$ respectively, with the appropriate indices.
With these definitions, a Killing spinor of valence $(k,l)$ is an element
$\varkappa_{A\cdots F}{}^{ A'\dots F'} \in \ker \sTwist_{k,l}$.

The fundamental operators appear naturally in the irreducible decomposition of the covariant derivative of a symmetric spinor field, see \cite[Lemma 15]{AnderssonBackdahlBlue}. In most calculations we freely make use of such decompositions. We shall make use of the following commutator relations for the fundamental operators.

\begin{lemma}[\protect{\cite[Lemma 18]{AnderssonBackdahlBlue}}]\label{lemma:commutators}
Let $\varphi_{AB} \in \mathcal{S}_{2,0}$, and $\varsigma$ a scalar. The operators $\sDiv$, $\sCurl$, $\sCurlDagger$ and $\sTwist$ satisfies the following commutator relations
\begin{subequations}
\begin{align}
(\sCurl_{1,1} \sTwist_{0,0} \varsigma)_{AB} ={}& 0,\label{eq:CurlTwist0}\\
(\sCurlDagger_{1,1} \sTwist_{0,0} \varsigma)_{A'B'} ={}& 0,\label{eq:CurlDaggerTwist0}\\
(\sCurl_{3,1} \sTwist_{2,0} \varphi)_{ABCD}={}&2 \Psi_{(ABC}{}^{F}\varphi_{D)F},\label{eq:CurlTwist}\\
(\sDiv_{3,1} \sTwist_{2,0} \varphi)_{AB}={}&
-  \tfrac{4}{3} (\sCurl_{1,1} \sCurlDagger_{2,0} \varphi)_{AB}
-8 \Lambda \varphi_{AB}
 + 2 \Psi_{ABCD} \varphi^{CD}.\label{eq:DivTwistCurlCurlDagger}
\end{align}
\end{subequations}
\end{lemma}

\subsection{Hardy estimate}\label{sec:Hardy}
Recall the scalar components of spinors transform under tetrad rescalings as sections of certain complex line bundles. The GHP operators $\tho, \tho', \edt, \edt'$ are defined	 on such scalars as the appropriate covariant derivative along the null tetrad legs $l^a,n^a, m^a, \bar{m}^a$, respectively, see \cite{Penrose:1986fk} for details.

\begin{lemma}\label{lem:hardy1} On any sphere with constant $r$ in the Schwarzschild spacetime, with $\varphi_{0}$ and $\varphi_{2}$ the extreme components of a smooth symmetric spinor field $\varphi_{AB}$ we have
\begin{subequations}
\begin{align}
\int_{S_r} |\varphi_{0}|^2 d\mu_{S_r}
\leq{}& r^2\int_{S_r} |\edt' \varphi_{0}|^2 \mbox{d}\mu_{S_r},\label{eq:HardyTheta0}\\
\int_{S_r} |\varphi_{2}|^2 d\mu_{S_r}
\leq{}& r^2\int_{S_r} |\edt \varphi_{2}|^2 \mbox{d}\mu_{S_r}.\label{eq:HardyTheta2}
\end{align}
\end{subequations}
\end{lemma}
\begin{proof}
Expand $\varphi_{0}$ and $\varphi_{2}$ in terms of spin-weighted spherical harmonics (see \cite[Section 4.15]{Penrose:1986fk})
\begin{subequations}
\begin{align}
\varphi_{0}(\theta,\phi)={}&\sum_{l=1}^\infty\sum_{m=-l}^{l} a_{l,m}\thinspace {}_1 Y_{l,m}(\theta,\phi), \\
\varphi_{2}(\theta,\phi)={}&\sum_{l=1}^\infty\sum_{m=-l}^{l} b_{l,m}\thinspace {}_{-1} Y_{l,m}(\theta,\phi). \end{align}
\end{subequations}
From \cite[Eq. (4.15.106)]{Penrose:1986fk} we have
\begin{subequations}
\begin{align}
\edt' \varphi_{0}(\theta,\phi)={}&\sum_{l=1}^\infty\sum_{m=-l}^{l} a_{l,m} \frac{\sqrt{l(l+1)}}{\sqrt{2}r}{}_0 Y_{l,m}(\theta,\phi),\\
\edt \varphi_{2}(\theta,\phi)={}&-\sum_{l=1}^\infty\sum_{m=-l}^{l} b_{l,m} \frac{\sqrt{l(l+1)}}{\sqrt{2}r}{}_{0} Y_{l,m}(\theta,\phi).
\end{align}
\end{subequations}
Through the orthogonality conditions \cite[Eq. (4.15.99)]{Penrose:1986fk} we get
\begin{subequations}
\begin{align}
\int_{S_r} |\varphi_{0}|^2 \mbox{d}\mu_{S_r}={}&4\pi \sum_{l=1}^\infty\sum_{m=-l}^{l} |a_{l,m}|^2,\\
\int_{S_r} |\varphi_{2}|^2 \mbox{d}\mu_{S_r}={}&4\pi \sum_{l=1}^\infty\sum_{m=-l}^{l} |b_{l,m}|^2,\\
\int_{S_r} |\edt' \varphi_{0}|^2 \mbox{d}\mu_{S_r}={}&4\pi \sum_{l=1}^\infty\sum_{m=-l}^{l} |a_{l,m}|^2 \frac{l(l+1)}{2r^2},\\
\int_{S_r} |\edt \varphi_{2}|^2 \mbox{d}\mu_{S_r}={}&4\pi \sum_{l=1}^\infty\sum_{m=-l}^{l} |b_{l,m}|^2 \frac{l(l+1)}{2r^2}.
\end{align}
\end{subequations}
This proves the desired inequalities.
\end{proof}

\section{A first order superenergy for the Maxwell field}
\subsection{General spacetime} In this subsection we shall work in a general spacetime with a Killing spinor $\kappa_{AB}$ of valence $(2,0)$ satisfying $(\sTwist_{2,0} \kappa)_{ABCA'} = 0$ and such that $\kappa_{AB} \kappa^{AB} \ne 0$. Recall that a Maxwell field $F_{ab}$ corresponds to a symmetric spinor of valence $(2,0)$ by
$$
F_{ab}={} \bar\epsilon_{A'B'} \phi_{AB}
 + \epsilon_{AB} \bar{\phi}_{A'B'}.
$$
The source-free Maxwell equation takes the form
$
(\sCurlDagger_{2,0} \phi)_{AA'} = 0 .
$
Instead of constructing energy-momentum tensors from the Maxwell field itself, we will use a first order expression.
In the paper \cite{AnderssonBackdahlBlue:NewConservationLaw:2014arXiv1412.2960A} we have introduced a conserved	 tensor for the Maxwell field on any spacetime admitting a valence $(2,0)$ Killing spinor with aligned matter. Here we shall take a different approach and construct a superenergy tensor satisfying the dominant energy condition which is however not conserved.
This tensor is constructed from a quantity $\beta_{AA'}$, which can be defined for any spacetime with an algebraically general valence $(2,0)$ Killing spinor. We remark that the leading order term in the tensor constructed here agrees with the tensor constructed in \cite{AnderssonBackdahlBlue:NewConservationLaw:2014arXiv1412.2960A}.

As we shall discuss later, the spinor fields introduced in the following definition correspond to the tensor fields given in the introduction.
 \begin{definition}\label{def:beta}
Assume that $\kappa_{AB}$ is a valence $(2,0)$ Killing spinor, such that $\kappa_{CD} \kappa^{CD}\neq 0$, $\phi_{AB}$ a solution to the source-free Maxwell equation $(\sCurlDagger_{2,0} \phi)_{AA'} = 0$, and define
\begin{subequations}
\begin{align}
\xi_{AA'}={}&(\sCurlDagger_{2,0} \kappa)_{AA'},\\
U_{AA'}={}&- \tfrac{1}{2} \nabla_{AA'}\log(-\kappa_{CD} \kappa^{CD}),\\
\Upsilon={}&\kappa^{AB} \phi_{AB},\\
\Theta_{AB}={}&-2 \kappa_{(A}{}^{C}\phi_{B)C},\\
\beta_{AA'}={}&(\sCurlDagger_{2,0} \Theta)_{AA'}
+U^{B}{}_{A'} \Theta_{AB}.
\end{align}
\end{subequations}
\end{definition}

\begin{lemma}\label{lem:betaprop}
Definition~\ref{def:beta} implies the relations
\begin{subequations}
\begin{align}
\beta_{AA'}={}&- U_{AA'} \Upsilon
 + (\sTwist_{0,0} \Upsilon)_{AA'},\label{eq:betaToUpsilon}\\
(\sDiv_{1,1} \beta)={}&- U^{AA'} \beta_{AA'},\label{eq:Divbeta}\\
(\sCurl_{1,1} \beta)_{AB}={}&U_{(A}{}^{A'}\beta_{B)A'},\label{eq:Curlbeta}\\
(\sCurlDagger_{1,1} \beta)_{A'B'}={}&U^{A}{}_{(A'}\beta_{|A|B')}.\label{eq:CurlDgbeta}
\end{align}
\end{subequations}
\end{lemma}
\begin{remark}
Observe that we have not assumed anything about the matter or the behaviour of $\xi_{AA'}$.
\end{remark}
\begin{proof}
Expanding the definition of $U_{AA'}$ yield
\begin{align}
U_{AA'}={}&\frac{2 \kappa_{AB} \xi^{B}{}_{A'}}{3 (\kappa_{CD} \kappa^{CD})}.
\end{align}
We can write the Maxwell field in terms of $\Theta_{AB}$ and $\Upsilon$ through the relation
\begin{align}
\kappa_{CD} \kappa^{CD} \phi_{AB}={}&- \Theta_{A}{}^{C} \kappa_{BC}
 + \kappa_{AB} \Upsilon.\label{eq:phiToUpsilon}
\end{align}
Differentiating the definition of $\Upsilon$, using the Killing spinor equation and the Maxwell equation gives
\begin{align}
\kappa^{BC} (\sTwist_{2,0} \phi)_{ABCA'}={}&\tfrac{2}{3} \xi^{B}{}_{A'} \phi_{AB}
 + (\sTwist_{0,0} \Upsilon)_{AA'}.\label{eq:kappaTwistphiToUpsilon}
\end{align}
Expanding the definition of $\beta_{AA'}$, using \eqref{eq:kappaTwistphiToUpsilon} and \eqref{eq:phiToUpsilon} we get
\begin{align}
\beta_{AA'}={}&U^{B}{}_{A'} \Theta_{AB}
 -  \tfrac{4}{3} \xi^{B}{}_{A'} \phi_{AB}
 + \kappa^{BC} (\sTwist_{2,0} \phi)_{ABCA'}\nonumber\\
={}&U^{B}{}_{A'} \Theta_{AB}
 -  \tfrac{2}{3} \xi^{B}{}_{A'} \phi_{AB}
 + (\sTwist_{0,0} \Upsilon)_{AA'}\nonumber\\
={}&- U_{AA'} \Upsilon
 + (\sTwist_{0,0} \Upsilon)_{AA'}.
\end{align}
Using the commutators \eqref{eq:DivTwistCurlCurlDagger}, \eqref{eq:CurlTwist} and the Killing spinor equation, we get
\begin{align}
(\sCurl_{1,1} \xi)_{AB}={}&(\sCurl_{1,1} \sCurlDagger_{2,0} \kappa)_{AB}
=-6 \Lambda \kappa_{AB}
 + \tfrac{3}{2} \Psi_{ABCD} \kappa^{CD},\label{eq:Curlxi}\\
0={}&\tfrac{1}{2} (\sCurl_{3,1} \sTwist_{2,0} \kappa)_{ABCD}
=\Psi_{(ABC}{}^{F}\kappa_{D)F}.\label{eq:IntCond}
\end{align}
Differentiating \eqref{eq:kappaTwistphiToUpsilon}, using \eqref{eq:phiToUpsilon}, the commutator \eqref{eq:DivTwistCurlCurlDagger}, the Maxwell equation, \eqref{eq:IntCond} and \eqref{eq:Curlxi}, we get a version of the Fackerell-Ipser equation
\begin{align}
(\sDiv_{1,1} \sTwist_{0,0} \Upsilon)={}&-4 \Lambda \Upsilon
 + \Psi_{ABCD} \kappa^{AB} \phi^{CD}
=\frac{2 \kappa^{AB} (\sCurl_{1,1} \xi)_{AB}}{3 (\kappa_{CD} \kappa^{CD})}\Upsilon.\label{eq:FI}
\end{align}
Direct calculations using the Killing spinor equation yield
\begin{align}
(\sDiv_{1,1} U)={}&- U_{AA'} U^{AA'}
 + \frac{2 \kappa^{AB} (\sCurl_{1,1} \xi)_{AB}}{3 (\kappa_{CD} \kappa^{CD})},&
(\sCurl_{1,1} U)_{AB}={}&0,&
(\sCurlDagger_{1,1} U)_{A'B'}={}&0.\label{eq:derU}
\end{align}
The equation \eqref{eq:Divbeta} follows from \eqref{eq:betaToUpsilon}, \eqref{eq:derU} and \eqref{eq:FI}.
The equations \eqref{eq:Curlbeta} and \eqref{eq:CurlDgbeta} follow from \eqref{eq:betaToUpsilon}, \eqref{eq:derU} and the commutators \eqref{eq:CurlTwist0} and \eqref{eq:CurlDaggerTwist0}.
\end{proof}
\begin{lemma}The components of $\phi_{AB}$, $\Theta_{AB}$ and $\beta_{AA'}$ in a principal dyad are related by
\begin{align}
\Theta_{0} ={}& -2\kappa_{1}\phi_{0}, &
\Theta_{1}={}&0, &
\Theta_{2}={}&2\kappa_{1}\phi_{2} ,\\
\beta_{00'}={}&\edt' \Theta_{0}, &
\beta_{01'}={}&\tho' \Theta_{0}, &
\beta_{10'}={}&- \tho \Theta_{2}, &
\beta_{11'}={}&- \edt \Theta_{2}.
\end{align}
\end{lemma}

The superenergy tensors for $\beta_{AA'}$ and $\Theta_{AB}$ are given by
\begin{subequations}
\begin{align}
\EMTensorH_{ABA'B'}={}&\tfrac{1}{2} \beta_{AB'} \overline{\beta}_{A'B}
 + \tfrac{1}{2} \beta_{BA'} \overline{\beta}_{B'A},  \label{eq:Hdef} \\
\mathbf{W}_{ABA'B'} ={}& \Theta_{AB} \overline\Theta_{A'B'} .
\end{align}
\end{subequations}
The notion of superenergy tensor extends to spinors of arbitrary valence, see \cite{senovilla:2000CQGra..17.2799S}.

We close this section by giving the correspondence between the spinor fields $\kappa_{AB}, \Theta_{AB}$ and the tensor fields $Y_{ab}, Z_{ab}$. We have
\begin{align*}
Y_{ab}={}&\tfrac{3}{2}i (\bar\epsilon_{A'B'} \kappa_{AB}
 -  \epsilon_{AB} \bar{\kappa}_{A'B'}) ,\\
Z_{ab}={}&\bar\epsilon_{A'B'} \Theta_{AB}
 + \epsilon_{AB} \overline{\Theta}_{A'B'} .
\end{align*}
The normalization of $Y_{ab}$ is chosen for convenience.
For the Schwarzschild spacetime with a principal dyad $(o_A,\iota_A)$ we have
\begin{align}\label{eq:kappaSchw}
\kappa_{AB}={}&\tfrac{2}{3} r o_{(A} \iota_{B)}.
\end{align}
In the Schwarzschild case, the above definitions with $\kappa_{AB}$ given by \eqref{eq:kappaSchw} yield the fields $Y_{ab}$, $\xi^a$, $U_a$, $Z_{ab}$, $\beta_a$,  $\EMTensorH_{ab}$, $\mathbf{W}_{ab}$ as in the introduction.
Finally,
the Morawetz current $\MorawetzCurrJ_{AA'}$ given in tensor form by \eqref{eq:MorawetzMomentum}
can be written in spinor form as
\begin{align}
\MorawetzCurrJ_{AA'}={}&\EMTensorH_{ABA'B'} \vecMprimary^{BB'}
 -  \tfrac{1}{2} \scalMprimary \bar{\beta}_{A'}{}^{B} \Theta_{AB}
 -  \tfrac{1}{2} \scalMprimary \beta_{A}{}^{B'} \overline{\Theta}_{A'B'}
  + \tfrac{1}{2} \Theta_{A}{}^{B} \overline{\Theta}_{A'}{}^{B'} (\sTwist_{0,0} \scalMprimary )_{BB'}.\label{eq:MorawetzCurrJDef}
\end{align}

\subsection{Schwarzschild spacetime}
\begin{lemma}\label{lem:divH}
For the Schwarzschild spacetime we have
\begin{subequations}\label{eq:divHeqs}
\begin{align}
\nabla^{BB'}\EMTensorH_{ABA'B'}={}&- U_{AA'} \beta^{BB'} \bar{\beta}_{B'B},\label{eq:divH}\\
\xi^{AA'} \nabla^{BB'}\EMTensorH_{ABA'B'}={}&0.\label{eq:xidivH}
\end{align}
\end{subequations}
In particular, $\xi^{BB'}\EMTensorH_{ABA'B'}$ is a positive conserved current.
\end{lemma}
\begin{proof}
For the Schwarzschild spacetime with a principal dyad $(o_A,\iota_A)$ the Killing spinor $\kappa_{AB}$ is given by \eqref{eq:kappaSchw}.
This gives immediately
\begin{subequations}
\begin{align}
\xi_{AA'}={}&(\partial_t)^{AA'}=\frac{(r - 2 M)^{1/2} \OrthFrameT^{AA'}}{r^{1/2}},\\
U_{AA'}={}&- r^{-1}\nabla_{AA'}r=\frac{(r - 2 M)^{1/2} \OrthFrameZ_{AA'}}{r^{3/2}}.
\end{align}
\end{subequations}

Computing the divergence of \eqref{eq:Hdef}, doing an irreducible decomposition of the derivatives, using Lemma~\ref{lem:betaprop} and using the reality of $U_{AA'}$ gives \eqref{eq:divH}. Equation \eqref{eq:xidivH}, then follows from $\xi^{AA'}U_{A'A}=0$.
\end{proof}
\begin{remark}
For the Kerr spacetime with non-vanishing angular momentum, the 1-form $U_{AA'}$ fails to be real and the current $\EMTensorH_{ab} \xi^b$ is not conserved.
\end{remark}

\section{Positive energy}\label{sec:posenergy}
Before we do any integrated decay estimates, we will verify that the energy \eqref{eq:ExiAqdef} will be positive and uniformly equivalent to the energy \eqref{eq:Exidef}.
\begin{theorem}\label{thm:PositiveEnergy}
Let $\vecMprimary^{AA'}$ and $\scalMprimary$ be given by \eqref{eq:Adef} and \eqref{eq:qdef}.
For any constant $|c_1|\leq 10/9$ and any spherically symmetric slice $\Sigma$ with future pointing timelike normal $N^{AA'}$ such that $N^{AA'}N_{AA'}=1$ we have a positive energy
\begin{align}
\int_{\Sigma} N^{AA'} (\EMTensorH_ {ABA'B'} \xi^{BB'} + c_ 1\MorawetzCurrJ_{AA'})  \mbox{d}\mu_{\Sigma_i} \geq{}&0.
\end{align}
\end{theorem}
\begin{proof}
We denote the superenergy tensor of a vector field by
\begin{align}
\SuperEnOne[\alpha_{AA'}]_{ABA'B'}={}&\tfrac{1}{2} \alpha_{AB'} \bar{\alpha}_{A'B}
 + \tfrac{1}{2} \alpha_{BA'} \bar{\alpha}_{B'A}.
\end{align}
This tensor satisfies the dominant energy condition.

Now, as the slice is spherically symmetric, the future pointing timelike vector field $N^{AA'}$ is spanned by $\OrthFrameT^{AA'}$ and $\OrthFrameZ^{AA'}$ with coefficients depending only on $r$. We can in general write it as
\begin{align}
N^{AA'} = (w(r) + \tfrac{1}{4} w(r)^{-1}) \OrthFrameT^{AA'} + (w(r) - \tfrac{1}{4} w(r)^{-1}) \OrthFrameZ^{AA'},
\end{align}
for some radial function $w(r)>0$.
We can then verify the identity
\begin{align}
\pm N^{AA'} (\bar{\beta}_{A'}{}^{B} \Theta_{AB} + \beta_{A}{}^{B'} \overline{\Theta}_{A'B'})={}&2 \varepsilon^{-1} \EMTensorH_{ABA'B'} N^{AA'} \OrthFrameT^{BB'}
 + 2 \varepsilon N^{AA'} \OrthFrameT^{BB'} \Theta_{AB} \overline{\Theta}_{A'B'}\nonumber\\
& - 2 \varepsilon^{-1} N^{AA'} \OrthFrameT^{BB'} \SuperEnOne[\beta_{AA'} \mp  b_1^{-1} \varepsilon N^{C}{}_{A'} \Theta_{AC}]_{ABA'B'}\nonumber\\
& -  b_1^{-1} \varepsilon \OrthFrameT^{AA'} \OrthFrameT^{BB'} \Theta_{AB} \overline{\Theta}_{A'B'},
\end{align}
where $b_1=N^{AA'} \OrthFrameT_{AA'} = \tfrac{1}{4} w(r)^{-1} + w(r)>0$.
The dominant energy condition then gives the Cauchy-Schwarz inequalities
\begin{align}
\pm N^{AA'} (\bar{\beta}_{A'}{}^{B} \Theta_{AB} + \beta_{A}{}^{B'} \overline{\Theta}_{A'B'})
\leq{}&2 \varepsilon^{-1} \EMTensorH_{ABA'B'} N^{AA'} \OrthFrameT^{BB'}
 + 2 \varepsilon N^{AA'} \OrthFrameT^{BB'} \Theta_{AB} \overline{\Theta}_{A'B'}.
\end{align}

Using this in the definition \eqref{eq:MorawetzCurrJDef} and the positivity of $q$, we get
\begin{align}
\pm N^{AA'} \MorawetzCurrJ_{AA'} \geq{}&
  \EMTensorH_{ABA'B'} N^{AA'} (\pm \vecMprimary^{BB'}
 -  \varepsilon^{-1} \scalMprimary \OrthFrameT^{BB'})\nonumber\\
& + N^{AA'} \Theta_{AB} \overline{\Theta}_{A'B'} \bigl(- \varepsilon \scalMprimary \OrthFrameT^{BB'}
 \pm \tfrac{1}{2} (\sTwist_{0,0} \scalMprimary )^{BB'}\bigr).
\end{align}

Expanding in the dyad gives
\begin{subequations}
\begin{align}
N^{AA'} \OrthFrameT^{BB'} \Theta_{AB} \bar{\Theta}_{A'B'}={}&w(r) |\Theta_{0}|^2
 + \frac{|\Theta_{2}|^2}{4 w(r)},\\
N^{AA'} \OrthFrameZ^{BB'} \Theta_{AB} \bar{\Theta}_{A'B'}={}&w(r) |\Theta_{0}|^2
 -  \frac{|\Theta_{2}|^2}{4 w(r)},\\
 N^{AA'} \OrthFrameT^{BB'} \EMTensorH_{ABA'B'}={}&w(r) |\edt' \Theta_{0}|^2
 + w(r) |\tho' \Theta_{0}|^2
 + \frac{|\tho \Theta_{2}|^2}{4 w(r)}
 + \frac{|\edt \Theta_{2}|^2}{4 w(r)}\nonumber\\
 \geq{}&w(r) |\edt' \Theta_{2}|^2
 + \frac{|\edt \Theta_{2}|^2}{4 w(r)}.
\end{align}
\end{subequations}
Hence, from \eqref{eq:HardyTheta0} and \eqref{eq:HardyTheta2} we get the Hardy estimates
\begin{subequations}
\begin{align}
\int_{S_r} N^{AA'} \OrthFrameT^{BB'} \Theta_{AB}\overline\Theta_{A'B'} d\mu_{S_r}
\leq{}&\frac{r^2}{2} \int_{S_r} N^{AA'} \OrthFrameT^{BB'} \EMTensorH_{ABA'B'} \mbox{d}\mu_{S_r},\\
\int_{S_r} |N^{AA'} \OrthFrameZ^{BB'} \Theta_{AB}\overline\Theta_{A'B'}| d\mu_{S_r}
\leq{}&\frac{r^2}{2} \int_{S_r} N^{AA'} \OrthFrameT^{BB'} \EMTensorH_{ABA'B'} \mbox{d}\mu_{S_r}.
\end{align}
\end{subequations}
Hence, we can make the estimates
\begin{align}
\int_{S_r} \pm N^{AA'} \MorawetzCurrJ_{AA'} \mbox{d}\mu_{S_r} \geq{}&
\int_{S_r} \EMTensorH_{ABA'B'} N^{AA'} (\pm\vecMprimary^{BB'} - B^{BB'})\mbox{d}\mu_{S_r},
\end{align}
where
\begin{align}
B^{AA'}={}&
   \varepsilon^{-1} \scalMprimary \OrthFrameT^{AA'}
 +  \tfrac{1}{2} \varepsilon \scalMprimary r^2 \OrthFrameT^{AA'}
  +  \tfrac{1}{4} |q'(r)| r^{3/2} (r - 2 M)^{1/2} \OrthFrameT^{AA'}.
\end{align}
Here we have used the ansatz $q=q(r)$, which yield
$(\sTwist_{0,0} \scalMprimary )_{AA'}=- (r - 2 M)^{1/2} r^{-1/2} q'(r) \OrthFrameZ_{AA'} $.
With our choice of $q(r)$ and
\begin{align}
\varepsilon={}&r^{- 3/2} (r - 2 M)^{1/2},
\end{align}
we get
\begin{align}
B^{AA'}={}& \tfrac{9}{8} M^2 r^{-4} \bigl(6 M^2
 - 13 M r
 + 6 r^2
 + |r - 3 M| (3 r - 5 M)\bigr) \xi^{AA'}.
\end{align}
In the region $r\geq 2M$, we have
\begin{align}
\tfrac{9}{8} M^2 r^{-4} \bigl(6 M^2
 - 13 M r
 + 6 r^2
 + |r - 3 M| (3 r - 5 M)\bigr) \leq \tfrac{2}{5}.
\end{align}
Hence
\begin{align}
\int_{S_r} \pm N^{AA'} \MorawetzCurrJ_{AA'} \mbox{d}\mu_{S_r} \geq{}&
\int_{S_r} \EMTensorH_{ABA'B'} N^{AA'} (\pm\vecMprimary^{BB'} - \tfrac{2}{5} \xi^{BB'})\mbox{d}\mu_{S_r}.
\end{align}
This gives
\begin{align}
\int_{S_r}N^{AA'} (\EMTensorH_ {ABA'B'} \xi^{BB'} + c_1 \MorawetzCurrJ_{AA'})  \mbox{d}\mu_{S_r} \geq{}&
\int_{S_r} \EMTensorH_{ABA'B'} N^{AA'} ( (1-\tfrac{2}{5}|c_1|) \xi^{BB'} + c_1 \vecMprimary^{BB'})\mbox{d}\mu_{S_r}.\label{ineq:Hc1J1}
\end{align}
With our value of $f(r)$, the vector field $ \xi^{AA'} + c_2 \vecMprimary^{AA'}$ is future pointing and timelike if $|c_2|\leq 2$. Hence, the right hand side of \eqref{ineq:Hc1J1} is non-negative if $|c_1|\leq 10/9$.
\end{proof}
\begin{corollary}\label{cor:uniformequivalence}
For any spherically symmetric slice $\Sigma$ with future pointing timelike normal $N^{AA'}$ such that $N^{AA'}N_{AA'}=1$ the energies $E_\xi(\Sigma)$ and
$E_{\xi +\vecMprimary, \scalMprimary}(\Sigma)$ are uniformly equivalent,
\begin{align}
\tfrac{1}{10}E_\xi(\Sigma)
\leq{}&
E_{\xi +\vecMprimary, \scalMprimary}(\Sigma)
\leq
\tfrac{19}{10}E_\xi(\Sigma).
\end{align}
\end{corollary}

\section{Integrated decay estimate}\label{sec:integrateddecay}
In this section, we will prove an integrated decay estimate.
\begin{lemma} \label{lem:Hardy-application}
On any sphere with constant $r$, we have
\begin{align}
\int_{S_r} \mathbf{W}_{\OrthFrameT\OrthFrameT}	d\mu_{S_r}
\leq{}&\frac{r^2}{2} \int_{S_r} |\beta_{\OrthFrameT}|^2 + |\beta_{\OrthFrameZ}|^2 \mbox{d}\mu_{S_r}.
\end{align}
\end{lemma}
\begin{proof}
We have
\begin{subequations}
\begin{align}
|\beta_{\OrthFrameZ}|^2 + | \beta_{\OrthFrameZ}|^2={}&|\edt \Theta_{2}|^2
 + |\edt' \Theta_{0}|^2,\label{eq:TbetaZbetatocomp}\\
\mathbf{W}_{\OrthFrameT\OrthFrameT}=\OrthFrameT^{AA'} \OrthFrameT^{BB'} \Theta_{AB}\overline\Theta_{A'B'}={}&
\tfrac{1}{2} |\Theta_{0}|^2
 + \tfrac{1}{2} |\Theta_{2}|^2.\label{eq:E2tocomp}
\end{align}
\end{subequations}
The equations \eqref{eq:TbetaZbetatocomp}, \eqref{eq:E2tocomp} and Lemma~\ref{lem:hardy1} gives the desired inequality.
\end{proof}

To analyse the positivity of the bulk term, we will write it in terms of quadratic forms. We shall use the following notation.
\begin{definition}
For any vector field $\nu^{AA'}$ we define the quadratic form
\begin{align}
E_1[b_0, b_1, b_2, \nu_{AA'}]={}&\bigl(2 b_2 \OrthFrameT^{A}{}_{A'} \OrthFrameT^{B}{}_{B'} + (b_1 - 2 b_2) \OrthFrameZ^{A}{}_{A'} \OrthFrameZ^{B}{}_{B'}\bigr) \nu_{(A}{}^{(B'}\bar{\nu}^{A')}{}_{B)}\nonumber\\
&- (b_0 + \tfrac{1}{4} b_1) \nu_{AA'} \bar{\nu}^{A'A},
\end{align}
where $b_0$, $b_1$, $b_2$ are scalar fields.
\end{definition}
In particular,
\begin{align} \label{eq:E1beta}
E_1[- \tfrac{1}{2},2,\tfrac{1}{2},\beta_{AA'}]={}&|\beta_{\OrthFrameT}|^2
 + |\beta_{\OrthFrameZ}|^2.
\end{align}
\begin{lemma}\label{lemma:E1pos}
$E_1[b_0, b_1, b_2, \nu_{AA'}]\geq 0$ for all $\nu^{AA'}$ if and only if
\begin{subequations}
\begin{align}
0 \leq{}& b_1,\\
0 \leq{}& b_2,\\
\max(- b_2, - b_1 + b_2) \leq{}& b_0 \leq b_2.
\end{align}
\end{subequations}
\end{lemma}
\begin{proof}
The eigenvalues of the quadratic form $E_1[b_0, b_1, b_2, \nu_{AA'}]$ at a point can be calculated by expanding in $\nu_{AA'}$ in dyad components. One finds the eigenvalues
\begin{equation} \label{eq:E1eig}
b_0 + b_1 -  b_2, \quad
- b_0 + b_2, \quad
b_0 + b_2, \quad
b_0 + b_2 .
\end{equation}
Requiring each eigenvalue in \eqref{eq:E1eig} to be non-negative gives the result.
\end{proof}

\begin{theorem}\label{thm:MorawetzEst}
Let $\MorawetzCurrJ_{AA'}$ be given by \eqref{eq:MorawetzCurrJDef} with the choices \eqref{eq:Adef} and \eqref{eq:qdef}. Then we have the estimate
\begin{align}
\int_{\Omega}-(\sDiv_{1,1} \MorawetzCurrJ)d\mu_{\Omega}\geq{}&
\int_{\Omega}\frac{1}{8}|\beta_{AA'}|^2_{1,\text{deg}}+\frac{M}{100 r^4}|\Theta_{AB}|^2_2 d\mu_{\Omega},
\end{align}
for any spherically symmetric spacetime region $\Omega$ of the Schwarzschild spacetime.
\end{theorem}
\begin{proof}
In order to motivate the choices \eqref{eq:Adef} and \eqref{eq:qdef}, we shall start by considering general radial $\vecMprimary^{AA'}$ and $\scalMprimary$.
From the form \eqref{eq:MorawetzCurrJDef}, the definition of $\beta_{AA'}$ and the properties from Lemma~\ref{lem:betaprop} we get
\begin{align}
- (\sDiv_{1,1} \MorawetzCurrJ)={}&-  \beta^{AA'} \bar{\beta}^{B'B} (\sTwist_{1,1} \vecMprimary)_{ABA'B'}
+\beta^{AA'} \bar{\beta}_{A'A} \bigl(
\tfrac{1}{4} (\sDiv_{1,1} \vecMprimary)
 + \vecMprimary^{BB'} U_{BB'}
- \scalMprimary\bigr)\nonumber\\
& + \Theta_{AB} \overline{\Theta}_{A'B'} \bigl(U^{AA'} (\sTwist_{0,0} \scalMprimary )^{BB'}
 -  \tfrac{1}{2} (\sTwist_{1,1} \sTwist_{0,0} \scalMprimary )^{ABA'B'}\bigr).\label{eq:divJJeq1}
\end{align}
With the ansatz
\begin{align}
\vecMprimary^{AA'}={}&f(r)(\partial_r)^{AA'}=\frac{f(r) r^{1/2} Z^{AA'}}{(r - 2 M)^{1/2}},&
\scalMprimary ={}&q(r),
\end{align}
we get
\begin{subequations}
\begin{align}
(\sDiv_{1,1} \vecMprimary)={}&\frac{2 f(r)}{r}
 + f'(r),\\
(\sTwist_{1,1} \vecMprimary)_{AB}{}^{A'B'}={}&\frac{f(r) \bigl(- (r - 3 M) \OrthFrameT_{(A}{}^{(A'}\OrthFrameT_{B)}{}^{B')} + (r - M) \OrthFrameZ_{(A}{}^{(A'}\OrthFrameZ_{B)}{}^{B')}\bigr)}{(r - 2 M) r}\nonumber\\
& -  \OrthFrameZ_{(A}{}^{(A'}\OrthFrameZ_{B)}{}^{B')} f'(r),\\
(\sTwist_{0,0} \scalMprimary )_{AA'}={}&- \frac{(r - 2 M)^{1/2} \OrthFrameZ_{AA'} q'(r)}{r^{1/2}},\\
(\sTwist_{1,1} \sTwist_{0,0} \scalMprimary )_{AB}{}^{A'B'}={}& \frac{(r - 3 M) (\OrthFrameT_{(A}{}^{(A'}\OrthFrameT_{B)}{}^{B')} -  \OrthFrameZ_{(A}{}^{(A'}\OrthFrameZ_{B)}{}^{B')}) q'(r)}{r^2}\nonumber\\
& +  \frac{(r - 2 M) \OrthFrameZ_{(A}{}^{(A'}\OrthFrameZ_{B)}{}^{B')} q''(r)}{r}.
\end{align}
\end{subequations}
We can now write the divergence of $J_{AA'}$ in terms of the quadratic forms $E_1$ and $\mathbf{W}_{\OrthFrameT\OrthFrameT}$, making use of the fact that since the middle component of $\Theta_{AB}$ is zero,	$\mathbf{W}_{\OrthFrameT\OrthFrameT} = \mathbf{W}_{\OrthFrameZ\OrthFrameZ}$. This gives
\begin{align}
- (\sDiv_{1,1} \MorawetzCurrJ)={}&E_1[q(r)
 + \frac{f(r) (r - M)}{2 r (r - 2 M)}
 -  \tfrac{1}{2} f'(r), - \frac{2 M f(r)}{(r - 2 M) r}
 + f'(r),\frac{f(r) (r - 3 M)}{2 (r - 2 M) r},\beta_{AA'}]\nonumber\\
& + \bigl (- \frac{(r - 2 M) q'(r)}{r^2}
 -  \frac{(r - 2 M) q''(r)}{2 r}\bigr )\mathbf{W}_{\OrthFrameT\OrthFrameT}.
\end{align}
Later we will need the Hardy estimate in the form of Lemma~\ref{lem:Hardy-application} to handle a negative contribution from the $\mathbf{W}_{\OrthFrameT\OrthFrameT}$ term. Therefore, we make use of \eqref{eq:E1beta} and extract a term
$g(r)\left(|\beta_{\OrthFrameT}|^2 + |\beta_{\OrthFrameZ}|^2\right)$
from the $E_1$ term, where $g(r)$ is a radial function to be chosen. This gives
\begin{align}
- (\sDiv_{1,1} \MorawetzCurrJ)={}&E_1[\tfrac{1}{2} g(r)
 + q(r)
 + \frac{f(r) (r - M)}{2 r (r - 2 M)}
-  \tfrac{1}{2} f'(r),-2 g(r)
 - \frac{2 M f(r)}{(r - 2 M) r}
 + f'(r),\nonumber\\
 &\qquad - \tfrac{1}{2} g(r)
 + \frac{f(r) (r - 3 M)}{2 (r - 2 M) r},\beta_{AA'}]
+\bigl(|\beta_{\OrthFrameT}|^2
 + |\beta_{\OrthFrameZ}|^2\bigr) g(r)\nonumber\\
& + \bigl (- \frac{(r - 2 M) q'(r)}{r^2}
 -  \frac{(r - 2 M) q''(r)}{2 r}\bigr )\mathbf{W}_{\OrthFrameT\OrthFrameT}.
 \label{eq:longdivPE1}
\end{align}
By Lemma \ref{lemma:E1pos}, the first terms are non-negative if and only if the inequalities
\begin{subequations}\label{eq:gqineqs}
\begin{align}
&0 \leq g(r)\leq \min
\Bigl(\tfrac{1}{2} f'(r)
 - \frac{M f(r)}{(r - 2 M) r},
  \frac{f(r) (r - 3 M)}{(r - 2 M) r}\Bigr),\\
&\max\Bigl(g(r)
 +  \frac{M f(r)}{(r - 2 M) r}
 -  \tfrac{1}{2} f'(r),
 - \frac{f(r)}{r}
 + \tfrac{1}{2} f'(r)\Bigr) \leq q(r) \leq
 - g(r)
 - \frac{M f(r)}{(r - 2 M) r}
 + \tfrac{1}{2} f'(r) ,
\end{align}
\end{subequations}
hold.
We clearly see that $f(r)$ must satisfy the condition $f(r)(r - 3M)\geq 0$. This means that it must change sign at $r=3M$. It is also important that it is bounded and vanishes at $r=2M$, so we can dominate $J_{AA'}$ with $\EMTensorH_{ABA'B'}\xi^{BB'}$. A natural choice is therefore
\begin{align}\label{eq:f(r)def}
f(r)={}&\frac{(r - 3 M) (r - 2 M)}{2 r^2}.
\end{align}
Now, $g(r)$ must be chosen such that
\begin{align}
0 \leq{}& g(r) \leq \min\Bigl(\frac{(r - 3 M)^2}{2 r^3}, \frac{3 M (r - 2 M)}{4 r^3}\Bigr).
\end{align}
We do not want to saturate the inequalities except at $r=2M$ and $r=3M$, and this can be achieved by setting
\begin{align}\label{eq:g(r)def}
g(r)={}&c_1 \frac{3 M (r - 3 M)^2 (r - 2 M)}{4 r^5},
\end{align}
for some constant $0<c_1<1$. Then, the upper bound for $q(r)$ is
\begin{align}
q(r) \leq \frac{3 M (r - 2 M) \bigl(r^2 - c_1 (r - 3 M)^2\bigr)}{4 r^5}.
\end{align}
We can therefore set
\begin{align}\label{eq:q(r)def}
q(r)={}&\frac{9 M^2 (r - 2 M) (2 r - 3 M)}{4 r^5}\geq 0.
\end{align}
A direct calculation shows that with $f(r)$, $g(r)$, and $q(r)$ given by \eqref{eq:f(r)def}, \eqref{eq:g(r)def}, and \eqref{eq:q(r)def}, respectively, the inequalities \eqref{eq:gqineqs} are satisfied everywhere and saturated only at $r=2M$ and $r=3M$. We now choose $c_1=5/6$. With these definitions we get using \eqref{eq:longdivPE1},
\begin{align}
- (\sDiv_{1,1} \MorawetzCurrJ)={}&
E_1[\frac{(r - 3 M)^2 (14 M^2 - 7 M r + 4 r^2)}{16 r^5},
\frac{M (90 M^3 - 105 M^2 r + 28 M r^2 + r^3)}{4 r^5},\nonumber\\
&\qquad\frac{(r - 3 M)^2 (10 M^2 - 5 M r + 4 r^2)}{16 r^5},\beta_{AA'}]\nonumber\\
& + \frac{5 M \bigl(|\beta_{\OrthFrameT}|^2 + |\beta_{\OrthFrameZ}|^2\bigr) (r - 3 M)^2 (r - 2 M)}{8 r^5}
 - \frac{27 M^2 (r - 5 M) (r - 2 M)^2}{2 r^8}\mathbf{W}_{\OrthFrameT\OrthFrameT}\nonumber\\
={}&
\frac{M |\beta_{\OrthFrameT}|^2 (r - 3 M)^2 (r - 2 M)}{8 r^5}
+ \frac{M |\beta_{\OrthFrameZ}|^2 (r - 2 M) (r^2 + 66 M r - 99 M^2)}{8 r^5}\nonumber\\
&+ \frac{\bigl(|\beta_{\OrthFrameX}|^2 + |\beta_{\OrthFrameY}|^2\bigr) (r - 3 M)^2 (2 r^2 - 3 M r + 6 M^2)}{4 r^5}\nonumber\\
& + \frac{5 M \bigl(|\beta_{\OrthFrameT}|^2 + |\beta_{\OrthFrameZ}|^2\bigr) (r - 3 M)^2 (r - 2 M)}{8 r^5}
 - \frac{27 M^2 (r - 5 M) (r - 2 M)^2}{2 r^8}\mathbf{W}_{\OrthFrameT\OrthFrameT}\nonumber\\
\geq{}&\frac{1}{8}|\beta_{AA'}|^2_{1,\text{deg}}
 + \frac{5 M \bigl(|\beta_{\OrthFrameT}|^2 + |\beta_{\OrthFrameZ}|^2\bigr) (r - 3 M)^2 (r - 2 M)}{8 r^5}\nonumber\\
& - \frac{27 M^2 (r - 5 M) (r - 2 M)^2}{2 r^8}\mathbf{W}_{\OrthFrameT\OrthFrameT}.
\label{eq:bigdivP}
\end{align}
The last term in \eqref{eq:bigdivP}
is not positive everywhere and therefore must be estimated.
Using the Hardy estimate in Lemma \ref{lem:Hardy-application}
we get
\begin{align}
\int_{S_r}-(\sDiv_{1,1} \MorawetzCurrJ)d\mu_{S_r}\geq{}&
\int_{S_r}\frac{1}{8}|\beta_{AA'}|^2_{1,\text{deg}}\nonumber\\
&\qquad+ \frac{M(r - 2 M)(5 r^3 - 84 M r^2 + 423 M^2 r - 540 M^3)}{4r^{8}} \mathbf{W}_{\OrthFrameT\OrthFrameT} d\mu_{S_r}\nonumber\\
\geq{}&\int_{S_r}\frac{1}{8}|\beta_{AA'}|^2_{1,\text{deg}}
+\frac{M (r - 2 M)}{100 r^5}\mathbf{W}_{\OrthFrameT\OrthFrameT} d\mu_{S_r}\nonumber\\
={}&\int_{S_r}\frac{1}{8}|\beta_{AA'}|^2_{1,\text{deg}}
+\frac{M}{100 r^4}|\Theta_{AB}|^2_2 d\mu_{S_r}.
\end{align}
\end{proof}

\section{Proof of the main theorem}\label{sec:mainthmproof}
\begin{proof}[Proof of Theorem~\ref{thm:MainResult}]
First of all Lemma~\ref{lem:divH}, and the divergence theorem gives that $E_{\xi}(\Sigma_2)=E_{\xi}(\Sigma_1)$.
Lemma~\ref{lem:divH}, Theorem~\ref{thm:MorawetzEst} and the divergence theorem yield
\begin{align}
\int_{\Omega} \frac{1}{8}\vert \beta_{a}\vert^2_{1,\text{deg}}
+\frac{M}{100 r^4}\vert Z_{ab}\vert^2_{2} d\mu_{\Omega}\leq {}&\int_{\Omega}-\nabla^a \MorawetzCurrJ_a d\mu_{\Omega}
= E_{\xi +\vecMprimary, \scalMprimary}(\Sigma_1) - E_{\xi +\vecMprimary, \scalMprimary}(\Sigma_2).
\end{align}
Corollary~\ref{cor:uniformequivalence} and the conservation of $E_{\xi}(\Sigma)$ then gives
\begin{align}
\int_{\Omega} \frac{1}{8}\vert \beta_{a}\vert^2_{1,\text{deg}}
+\frac{M}{100 r^4}\vert Z_{ab}\vert^2_{2} d\mu_{\Omega}
\leq {}& \frac{19}{10}E_{\xi}(\Sigma_1) - \frac{1}{10}E_{\xi}(\Sigma_2)
=\frac{9}{5}E_{\xi}(\Sigma_1).
\end{align}
\end{proof}

\section*{Acknowledgements} We are grateful to Steffen Aksteiner for discussions and helpful remarks.  PB and TB were supported by EPSRC grant EP/J011142/1. LA was supported in part by the Knut and Alice Wallenberg Foundation, under a contract to the Royal Institute of Technology, Stockholm, Sweden.

\newcommand{\arxivref}[1]{\href{http://www.arxiv.org/abs/#1}{{arXiv.org:#1}}}
\newcommand{\mnras}{Monthly Notices of the Royal Astronomical Society}
\newcommand{\prd}{Phys. Rev. D}

\end{document}